\documentclass[10pt]{article}
\usepackage{amsmath,amssymb,amsbsy,amsfonts,amsthm,latexsym,
                        amsopn,amstext,amsxtra,euscript,amscd}
\begin{document}
\newtheorem{theorem}{Theorem}
\newtheorem{lemma}[theorem]{Lemma}
\newtheorem{claim}[theorem]{Claim}
\newtheorem{cor}[theorem]{Corollary}
\newtheorem{prop}[theorem]{Proposition}
\newtheorem{definition}{Definition}
\newtheorem{question}[theorem]{Open Question}

\def\cA{{\mathcal A}}
\def\cB{{\mathcal B}}
\def\cC{{\mathcal C}}
\def\cD{{\mathcal D}}
\def\cE{{\mathcal E}}
\def\cF{{\mathcal F}}
\def\cG{{\mathcal G}}
\def\cH{{\mathcal H}}
\def\cI{{\mathcal I}}
\def\cJ{{\mathcal J}}
\def\cK{{\mathcal K}}
\def\cL{{\mathcal L}}
\def\cM{{\mathcal M}}
\def\cN{{\mathcal N}}
\def\cO{{\mathcal O}}
\def\cP{{\mathcal P}}
\def\cQ{{\mathcal Q}}
\def\cR{{\mathcal R}}
\def\cS{{\mathcal S}}
\def\cT{{\mathcal T}}
\def\cU{{\mathcal U}}
\def\cV{{\mathcal V}}
\def\cW{{\mathcal W}}
\def\cX{{\mathcal X}}
\def\cY{{\mathcal Y}}
\def\cZ{{\mathcal Z}}

\def\A{{\mathbb A}}
\def\B{{\mathbb B}}
\def\C{{\mathbb C}}
\def\D{{\mathbb D}}
\def\E{{\mathbb E}}
\def\F{{\mathbb F}}
\def\G{{\mathbb G}}
\def\I{{\mathbb I}}
\def\J{{\mathbb J}}
\def\K{{\mathbb K}}
\def\L{{\mathbb L}}
\def\M{{\mathbb M}}
\def\N{{\mathbb N}}
\def\O{{\mathbb O}}
\def\P{{\mathbb P}}
\def\Q{{\mathbb Q}}
\def\R{{\mathbb R}}
\def\S{{\mathbb S}}
\def\T{{\mathbb T}}
\def\U{{\mathbb U}}
\def\V{{\mathbb V}}
\def\W{{\mathbb W}}
\def\X{{\mathbb X}}
\def\Y{{\mathbb Y}}
\def\Z{{\mathbb Z}}

\def\E{{\mathbf E}}
\def\Fp{\F_p}
\def\ep{{\mathbf{e}}_p}
\def\Nm{{\mathrm{Nm}}}
\def\lcm{{\mathrm{lcm}}}

\def\scr{\scriptstyle}
\def\\{\cr}
\def\({\left(}
\def\){\right)}
\def\[{\left[}
\def\]{\right]}
\def\<{\langle}
\def\>{\rangle}
\def\fl#1{\left\lfloor#1\right\rfloor}
\def\rf#1{\left\lceil#1\right\rceil}
\def\le{\leqslant}
\def\ge{\geqslant}
\def\eps{\varepsilon}
\def\mand{\qquad\mbox{and}\qquad}
\newcommand{\ignor}[1]{}
\newcommand{\comm}[1]{\marginpar{%
\vskip-\baselineskip 
\raggedright\footnotesize
\itshape\hrule\smallskip#1\par\smallskip\hrule}}

\def\xxx{\vskip5pt\hrule\vskip5pt}
\def\Mat{\mathop{Mat}}
\def\gcd{\mathop{gcd}}
\def\rank{\mathop{rank}}
\def\trace{\mathop{trace}}
\def\ind{\mathrm{ind}}
\def\IM{\mathrm{Im}}

\title{{\bf Cycles in Repeated Exponentiation 
Modulo $p^n$.}}
\author{
         {\sc{Lev Glebsky}} \\
         {Instituto de Investigaci{\'o}n en Comunicacin {\'O}ptica}\\   
          {Universidad Aut{\'o}noma de San Luis Potos{\'i}} \\
         {Av. Karakorum 1470, Lomas 4a 78210}\\
          {San Luis Potosi, Mexico} \\
         {\tt glebsky@cactus.iico.uaslp.mx}}
\date{\today}
\pagenumbering{arabic}

\maketitle

\begin{abstract}
Given a number $r$, we consider the dynamical system generated by 
repeated exponentiations modulo $r$, that 
is,  by the map  $u \mapsto f_g(u)$, where 
$f_g(u) \equiv g^u \pmod r$ and $0 \le f_g(u) \le r-1$. 
The number of cycles of the defined above dynamical system 
is considered for $r=p^n$. 
\end{abstract}

\section{Introduction and formulation of results} 
Given a number $r$, we consider the dynamical system generated by 
repeated exponentiations modulo $r$, that 
is,  by the map  $u \mapsto f_q(u)$, where 
$f_q(u) \equiv q^u \pmod r$ and $0 \le f_q(u) \le r-1$.
In \cite{GleShpar} the author with Igor Shparlinski considered the case where
$r$ is a prime. We gave some estimates on number of 
$1-,2-,3-$periodic
points of $f$. We believe that our estimates are very far from being strict
 (but it seems that
the better estimates are not known). Maybe one of the difficulties 
of the problem 
is that $f$ is not an algebraic factor of $q^x$: 
if, for example, $\gcd(r,\phi(r))=1$ then 
one can choose representative $y\equiv x\mod r$ such that $q^y$ 
has any possible value
$\mod r$. The situation where $\gcd(r,\phi(r))$ is large may be more easy 
to deal with.
In that case, instead of considering the function $f$, 
one may consider the graph 
with edges from $x\in\Z_r$ to all $q^y \mod r$, $y\equiv x \mod r$. 
I will show that it works very well at list for $r=p^n$ with a prime $p$. 
In what follows we will suppose that $\gcd(q,p)=1$.

Let $\Gamma_{p,n,q}$ be a directed graph defined as follows:
the set of vertexes is $V(\Gamma)=\Z_{p^n}$ and the set of edges is
$E=\{(x,q^y \mod p^n)\;|\;x\in\Z_{p^n}, y \equiv x \mod p^n\}$.  
Suppose for a moment that $q$ is primitive $\mod p^n$. Then $p-1$ is
the out degree of any edge of the graph $\Gamma$.
Let $C_{p,n,q}(k)$ be the number of k-cycles (with initial vertex marked) 
in $\Gamma_{p,n,q}$. 
\begin{theorem}\label{th_main1}
$C_{p,n,q}\leq (p-1)^k$. If $q$ is primitive $\mod p$ then
$C_{p,n,q}= (p-1)^k$.
\end{theorem}
\begin{cor}
The number of $k$-periodic points for $f(x)\equiv q^x\mod p^n$,
$0 \leq f(x)<p^n$ is less than  
$(p-1)^k$. 
\end{cor}

The same technique may be used to estimate the number of $k$-cyclic
points in  ``additive perturbations'' of graph $\Gamma$.
Precisely, let us define $\Gamma^{+r}_{p,n,q}$ as follows:
the set of vertexes is $V(\Gamma)=\Z_{p^n}$ and the set of edges is
$E=\{(x,q^y+c \mod p^n)\;|\;x\in\Z_{p^n}, y \equiv x 
\mod p^n\, c=-r,-r+1,\dots,r\}$.  
Let $C^{+r}_{p,n,q}(k)$ be the number of k-cycles 
(with the initial vertex marked) 
in $\Gamma^{+r}_{p,n,q}$.
\begin{theorem}\label{th_main2}
$C^{r+}_{p,n,q}(k)\leq p+rp[2p(2r+1)]^{k}(n-1)$
\end{theorem} 
So, $C$ grows no more than linearly in $n$ (but the number of all 
vertexes grows exponentially). 
\section{Proof of Theorem~\ref{th_main1}}
\begin{lemma} \label{lm1}
Let $A_1,A_2,...,A_r$ be elements of an associative 
(not necessarily commutative)
algebra $\cA$. Let $M\in\Mat_{n\times n}(\cA)$,
$$
M=\(\begin{array}{cccc} A_1 & A_2 &\dots & A_n\\
                       A_1 & A_2 &\dots & A_n\\
                    \vdots &\vdots &\dots &\vdots\\
                       A_1 & A_2 &\dots & A_n
     \end{array}\) 
$$
Then $\trace(M^k)=(A_1+A_2+\dots+A_r)^k$.
\end{lemma}
\begin{proof}
$$
M^k=\(\begin{array}{c}1\\
                    1\\
                \vdots\\
                    1
\end{array}\)
\(\(\begin{array}{cccc} A_1 & A_2 & \dots & A_r
\end{array}\)
\(\begin{array}{c}1\\
                    1\\
                \vdots\\
                    1
\end{array}\)\)^{(k-1)}
\(\begin{array}{cccc} A_1 & A_2 & \dots & A_r
\end{array}\) =
$$
$$
(A_1+A_2+\dots+A_r)^{k-1}\(\begin{array}{cccc} A_1 & A_2 &\dots & A_n\\
                       A_1 & A_2 &\dots & A_n\\
                    \vdots &\vdots &\dots &\vdots\\
                       A_1 & A_2 &\dots & A_n
     \end{array}\) 
$$
\end{proof}
\begin{lemma} \label{lm2}
Let $A_n$ be the adjacency matrix of $\Gamma_{p,n,q}$. Then
\begin{enumerate}
\item $A_1=\(\begin{array}{ccccc} 0 & 1 & 1 & \dots & 1 \\
                                  0 & 1 & 1 & \dots & 1 \\
                            \vdots &\vdots &\vdots &\dots & \vdots\\
                                  0 & 1 & 1 & \dots & 1
                     \end{array}\) $, if $q$ is primitive $\mod p$. 
If $q$ is not primitive then $A_1$ has the same form with some $1$ changed to
$0$. 

\item for $n>1$ $A_n=\(\begin{array}{cccc} B_1^n & B_2^n &\dots & B_p^n \\
                                           B_1^n & B_2^n &\dots & B_p^n \\
                                       \vdots &\vdots & \dots &\vdots \\
                                           B_1^n & B_2^n &\dots & B_p^n \\
                       \end{array}\)$,
where $B_j^n\in\Mat_{p^{n-1}\times p^{n-1}}(\Z)$ and $B_1^n+B_2^n+...+B_p^n=A_{n-1}$. 
\end{enumerate}
\end{lemma}
\begin{proof}
Item 1 is trivial. Let us prove Item 2. 
First of all we represent $x\in\Z_{p^n}=\{0,1,2,\dots,p^n-1\}$ as
$x=y+bp^{n-1}$, where $y\in\{0,1,\dots,p^{n-1}-1\}$ and $b\in\{0,1,\dots,p-1\}$.
The block structure of $A_n$ corresponds to the described above representation,
such that $b$'s are numbering our blocks and $y$'s are numbering the elements inside 
the blocks.
The item 2 follows from the next facts
\begin{description}
\item[i)] $O^n(x)=O^n(y)$ if $x\equiv y\mod p^{n-1}$. Where
$O^n(x)=\{y\in Z_{p^n}\;|\;(x,y)\in E(\Gamma_{n,p,q})\}$.
\item[ii)] Let $\phi:\Z_{p^n}\to\Z_{p^{n-1}}$ be defined as 
$\phi(x)\equiv x \mod p^{n-1}$.

Then for any  
$y\in\{0,1,2,\dots,p^{n-1}-1\}$   
$\phi$ defines a bijection $O^{n}(y)\leftrightarrow O^{n-1}(y)$.  
\end{description}

Fact {\bf i)}. To find $q^z \mod p^n$ it suffices to know $z \mod (p-1)p^{n-1}$.
Let $P_x=\{z\in \Z_{(p-1)p^{n-1}}\;|\;\exists a\in\Z\; a\equiv z \mod (p-1)p^{n-1}
\mbox{ and } a\equiv x \mod p^n\}$. One has that 
$O^n(x)=\{q^z\mod p^n\;|\; z\in P_x\}$. By Chinese Remainder Theorem
$P_x=P_y$ if and only if $x\equiv y\mod p^{n-1}$. 
Observe that $O^n(x)=\{q^xq^{bp^n} \mod p^n\;|\;b\in\{0,1,\dots,p-2\}\}$. 

Fact {\bf ii)}. 
Recall that $O^n(x)=\{q^xq^{bp^n}\mod p^n\;|\;b\in\{0,1,\dots,p-2\}\}$ and
$O^{n-1}(x)=\{q^xq^{bp^{n-1}}\mod p^{n-1}\;|\;b\in\{0,1,\dots,p-2\}\}$.
Now, $q^{bp^{n-1}}\equiv q^{bp^n}\mod p^n$. Indeed, $bp^{n-1}-bp^n\equiv 0 \mod (p-1)p^{n-1}$.
It proves fact {\bf ii)} if $q$ is primitive $\mod p^{n-1}$. For non primitive $q$
it suffices to prove that for $b_1,b_2\in\{0,1,\dots,p-2\}$ the  congruence
\begin{equation}\label{congr1}
q^{b_1p^{n-1}}\equiv q^{b_2p^{n-1}}\mod p^{n-1}
\end{equation}
imply the congruence
\begin{equation}\label{congr2}
q^{b_1p^{n-1}}\equiv q^{b_2p^{n-1}}\mod p^{n}
\end{equation}
Let $q\equiv g^r\mod p^n$ for primitive $g$. The first congruence is
equivalent to $(b_1-b_2)rp^{n-1}\equiv 0\mod (p-1)p^{n-2}$. It implies
$(p-1)|(b_1-b_2)r$. So, $(b_1-b_2)rp^{n-1}\equiv 0\mod (p-1)p^{n-1}$ and
the second congruence follows.
\end{proof}

Now it is easy to finish the proof of Theorem~\ref{th_main1}.
First of all $C_{p,n,q}(k)=\trace((A_n)^k)$. Using Lemma~\ref{lm1}, 
Lemma~\ref{lm2} and compatibility of the trace and multiplication 
with the block structure
we get
$$
\trace((A_n)^k)=\trace((A_{n-1})^k)=\dots = \trace((A_1)^k)=(p-1)^k
$$ 

\section{Proof of theorem~\ref{th_main2}}
For $A,B\in\Mat_{d\times d}(\{0,1\})$ we will write $A\preceq B$
if $A_{i,j}=1$ implies $B_{i,j}=1$.
\begin{lemma}\label{new_lm2}
Let $A_n$ be the adjacency matrix of $\Gamma^{+r}_{p,n,q}$. Then
\begin{enumerate}
\item $A_1\preceq\(\begin{array}{ccccc} 1 & 1 & 1 & \dots & 1 \\
                                  1 & 1 & 1 & \dots & 1 \\
                            \vdots &\vdots &\vdots &\dots & \vdots\\
                                  1 & 1 & 1 & \dots & 1
                     \end{array}\) $, if $q$ is primitive $\mod p$. 
If $q$ is not primitive then $A_1$ has the same form with some $1$ changed to
$0$. 

\item for $n>1$ $A_n\preceq\(\begin{array}{cccc} B_1^n & B_2^n &\dots & B_p^n \\
                                           B_1^n & B_2^n &\dots & B_p^n \\
                                       \vdots &\vdots & \dots &\vdots \\
                                           B_1^n & B_2^n &\dots & B_p^n \\
                       \end{array}\)+X$,
where $B_j^n\in\Mat_{p^{n-1}\times p^{n-1}}(\Z)$, $B_1^n+B_2^n+...+B_p^n=A_{n-1}$,
$X\in\Mat_{p^n\times p^n}(\{0,1\}$ with less then $2rp$ rows. 
\end{enumerate}
\end{lemma}
\begin{proof}
Item 1 is trivial. The prove of Item 2 proceeds the same way as the one 
of Theorem~\ref{th_main1}, but now we have to take into account that 
$y+s (\mod p^{n-1})$ may be different from $y+s (\mod p^{n})$. Observe, that
$y+s (\mod p^{n-1})=y+s (\mod p^{n})$
for $r \geq y\leq p^{n-1}-1-r$. So, for each $b\in \{0,1,\dots,p-1\}$ 
there exists
only $2r$ of $y\in\{0,1,\dots,p^{n-1}-1\}$ where the rows of $X$ are non zero.   
\end{proof}
Now we are ready to prove Theorem~\ref{th_main2}. 
$$
C^{+r}_{n,p,q}(k)=c_n=\trace(A_n^k)\leq\trace(A_{n-1}^k)+\Delta=c_{n-1}+\Delta
$$
$\Delta$ is the sum of the traces of $2^{k-1}$ matrices $P_s$, each of them is a product
of $k$ matrices containing $X$. Observe, that
$\trace(P_s)\leq 2rp((2r+1)p)^k$. Indeed,  this is a number of $k$-periodic paths such that
some steps of the path correspond to the matrix  $X$ and some to the matrix $B$.
The estimate follows from the number of non-zero rows of $X$, and that each row of $X$ and
$B$ contains no more than $(2r+1)p$ ones. Noting that $c_1\leq p$ we get
$c_n\leq p+rp(2(2r+1)p)^k(n-1)$.

{\bf Acknowledgments} The proof of Lemma~\ref{lm1} was suggested by Edgardo Ugalde.
The author thanks Igor Shparlinski for useful suggestions. The work were partially 
supported by PROMEP grant UASLP-CA21 and by
CONACyT grant 50312.

\ignor{
\bibitem{BKS} J.~Bourgain, S. V. Konyagin and
I. E. Shparlinski,
`Product sets of rationals, multiplicative
translates of subgroups in residue rings
and fixed points of the discrete logarithm',
{\it Intern.\ Math.\ Research Notices\/},
{\bf 2008} (2008), Article ID rnn090, 1--29
(Corrigenda {\it Intern.\ Math.\ Research Notices\/},
{\bf 2009} (2009), 3146-3147).

\bibitem{CobZah} C. Cobeli and A. Zaharescu, `An exponential congruence with
solutions in primitive roots', {\it Rev. Roumaine Math. Pures Appl.\/},
{\bf 44} (1999),   15--22.

\bibitem{CoSh} D. Coppersmith and I. E. Shparlinski, `On polynomial approximation of the
discrete logarithm and the Diffie--Hellman mapping',
{\it J. Cryptology\/}, {\bf 13} (2000),  339--360.

\bibitem{GoldRos} O. Goldreich and V. Rosen, `On the security
of modular exponentiation with
application to the construction of pseudorandom generators',
{\it J.  Cryptology\/}, {\bf 16} (2003), 71--93.

\bibitem{Hold} J. Holden, `Fixed points and two cycles of the discrete logarithm',
 {\it Lect. Notes in Comp. Sci.\/}, Springer-Verlag, Berlin,   {\bf 2369}
(2002), 405--416.

\bibitem{HoldMor1} J. Holden and P. Moree, `New conjectures and results for small
cycles of the discrete logarithm',
{\it High Primes and Misdemeanours: Lectures in Honour of
the 60th Birthday of Hugh Cowie Williams\/},
Fields Institute Communications, vol.41, Amer. Math. Soc., 2004, 245--254.

\bibitem{HoldMor2} J. Holden and P. Moree, `Some heuristics and
 and results for small cycles of the discrete
logarithm', {\it Math. Comp.\/},  {\bf 75} (2006), 419--449.

\bibitem{MeWi} G. C. Meletiou and A. Winterhof,
`Interpolation of the double discrete logarithm',
{\it Lect. Notes in Comp. Sci.\/}, Springer-Verlag, Berlin,   
{\bf  5130} (2008), 1--10.

\bibitem{PaSu} S. Patel and G. S. Sundaram,
 `An efficient discrete $\log$ pseudo random generator',
{\it  Lect. Notes in Comp. Sci.\/}, Springer-Verlag, Berlin, {\bf 1462} (1999), 35--44.

\bibitem{Silv} J.~H.~Silverman, {\it The arithmetic of elliptic
curves\/},
Springer-Verlag, Berlin, 1995.

\bibitem{Zhang} W. P. Zhang,
`On a problem of Brizolis',
{\it Pure Appl. Math.\/}, {\bf 11} (1995), suppl., 1--3 (in Chinese).
 
\end{thebibliography}
}
\end{document}